%% file: main.tex
\documentclass[submission,copyright,creativecommons,sharealike,noncommercial]{eptcs}
\providecommand{\event}{QPL 2018}
\pdfoutput=1

\input{packages}

\input{environments}

\input{macros}

\input{newmacros}

\title{Higher-order CPM Constructions}
\author{
	Stefano Gogioso\\
	University of Oxford \\
	\texttt{stefano.gogioso@cs.ox.ac.uk}
}

\def\titlerunning{Higher-order CPM Constructions}
\def\authorrunning{Stefano Gogioso}

\begin{document}

\maketitle
\vspace{-2mm}

\begin{abstract}
	We define a higher-order generalisation of the CPM construction based on arbitrary finite abelian group symmetries of symmetric monoidal categories. We show that our new construction is functorial, and that its closure under iteration can be characterised by seeing the construction as an algebra for an appropriate monad. We provide several examples of the construction, connecting to previous work on the CPM construction and on categorical probabilistic theories, as well as upcoming work on higher-order interference and hyper-decoherence.
\vspace{-2mm}
\end{abstract}

\section{Introduction}

The CPM Construction \cite{selinger2007dagger} is of cardinal importance to the categorical study of quantum theory \cite{abramsky2004categorical,coecke2017picturing}, where it provides the canonical model of mixed-state quantum behaviour. It has been extensively studied, both axiomatically \cite{coecke2008axiomatic,coecke2016pictures,coecke2010environment,heunen2013completely} and concretely \cite{gogioso2015bestiary,marsden2017ambiguity,marsden2015graph,piedeleu2015open}. Recently, applications of the CPM construction in the context of compositional distributional models of meaning\cite{balkir2015distributional,bankova2016graded,marsden2017ambiguity,piedeleu2015open} have prompted renewed interest on iterated CPM constructions \cite{ashoush2016dual}, with the discovery of new features due to their additional degrees of freedom \cite{zwart2017double}. 

In this work, we define a theory of higher-order 
\footnote{By higher-order we mean that the abelian symmetry groups associated with our CPM constructions can have any finite order, as opposed to the order-2 symmetry group $\integersMod{2}$ associated with the traditional CPM construction on dagger-compact categories. }
CPM constructions, which we characterise as Eilenberg-Moore algebras for a certain monad. We connect to the recent work on iterated CPM constructions \cite{zwart2017double}. We provide a very broad family of examples obtained from categories of free finite-dimensional modules over commutative semirings \cite{gogioso2017fantastic}, showing that they can all be understood within the framework of categorical probabilistic theories \cite{gogioso2017categorical}.

\vspace{-2mm}
\section{The traditional CPM construction}

In the traditional formulation of \cite{selinger2007dagger} and subsequent work, the CPM construction can be understood in terms of two separate steps: \emph{doubling} and \emph{discarding}. By \emph{doubling}, we mean the passage from a dagger-compact category $\CategoryC$ to the corresponding doubled category $\DoubledCategory{\CategoryC}$. By \emph{discarding}, we mean the introduction of an environment structure $(\trace{A})_{A \in \obj{\CategoryC}}$ into the doubled category.

\subsection{Doubling}

The doubling step of the traditional CPM construction can be understood as the passage from $\CategoryC$ to the sub-category $\DoubledCategory{\CategoryC}$ of $\CategoryC$ obtained as the image of the following doubling functor:
\[
\begin{array}{ccc}
	\DoublingFunctor{A} & := & A \otimes A^\ast \\
	\DoublingFunctor{f} & := & f \otimes f^\ast
\end{array}
\]
This work uses a conjugation functor $(\emptyArg)^\ast$ which is both strict monoidal---i.e. $(A\otimes B)^\ast = (A^\ast \otimes B^\ast)$---and involutive---i.e. $\big((\emptyArg)^\ast\big)^\ast = \id{}$. The $\DoublingFunctor{\emptyArg}$ functor above respects the dagger, but if $\DoubledCategory{\CategoryC}$ is equipped with the tensor product inherited from $\CategoryC$ then the functor is not strict monoidal:
\[
	\DoublingFunctor{A\otimes B}
	=
	(A \otimes B)\otimes(A \otimes B)^\ast 
	=
	(A \otimes B)\otimes(A^\ast \otimes B^\ast) 
	\neq
	(A \otimes A^\ast)\otimes(B \otimes B^\ast)
	=
	\DoublingFunctor{A} \otimes \DoublingFunctor{B}
\]
The functor $\DoublingFunctor{\emptyArg}$ becomes strict monoidal if we equip the image with a different tensor product---analogous to the one of \cite{selinger2007dagger} on morphisms---which respects the structure of objects/morphisms:
\[
\begin{array}{ccccc}
	(A \otimes A^\ast) \boxtimes (B \otimes B^\ast) 
	& := &
	(A \otimes B) \otimes (A^\ast \otimes B^\ast)
	& = &
	(A \otimes B) \otimes (A \otimes B)^\ast
	\\
	(f \otimes f^\ast) \boxtimes (g \otimes g^\ast) 
	& := &
	(f \otimes g) \otimes (f^\ast \otimes g^\ast)
	& = &
	(f \otimes g) \otimes (f \otimes g)^\ast
\end{array}
\]
From this moment forward, when saying ``monoidal'' we will always mean ``strict monoidal''. Under this new tensor product, $\DoubledCategory{\CategoryC}$ is a dagger SMC, and the doubling functor is dagger monoidal . 

\vspace{-2mm}
\subsection{Discarding}

The discarding step of the traditional CPM construction can be understood as the introduction of a family of effects $(\trace{A}: A \otimes A^\ast \rightarrow I)_{A \in \obj{\CategoryC}}$ in $\CategoryC$ (the \emph{discarding maps}) which respect the tensor product of $\DoubledCategory{\CategoryC}$ in the following sense:
\[
\begin{array}{rcl}	
	\trace{A \otimes B} & = & (\trace{A} \otimes \trace{B}) \circ (\id{A} \otimes \sigma_{B,A^\ast} \otimes \id{B^\ast}) \\
	\trace{I} &=& 1
\end{array}
\]
where $\sigma_{A,B}:A \otimes B \rightarrow B \otimes A$ are the symmetry isomorphisms of $\CategoryC$. Traditionally, the chosen family is $\trace{A} := \varepsilon_{A}$, where by $\eta_{A}:I \rightarrow A^\ast \otimes A$ and $\varepsilon_{A}: A \otimes A^\ast \rightarrow I$ we will denote the cups and caps for the compact closed structure of $\CategoryC$.

Given the above, the CPM category $\CPMFunctor{\CategoryC}$ can be defined as the smallest sub-category of $\CategoryC$ containing the doubled category $\DoubledCategory{\CategoryC}$ and all the effects $(\trace{A})_{A \in \obj{\CategoryC}}$.
\footnote{This CPM category is equivalent to the one originally defined in \cite{selinger2007dagger}, but not identical: the objects in the original $\CPMFunctor{\CategoryC}$ are labelled by the objects $A$ of $\CategoryC$, while here $\CPMFunctor{\CategoryC}$ is defined directly as a sub-category of $\CategoryC$ with objects in the form $A \otimes A^\ast$.}
Because the effects are required to respect the monoidal structure of $\DoubledCategory{\CategoryC}$, the CPM category is itself a symmetric monoidal category with tensor product $\boxtimes$ extended as follows to arbitrary morphisms $F:\DoublingFunctor{A} \rightarrow \DoublingFunctor{C}$ and $G:\DoublingFunctor{B} \rightarrow \DoublingFunctor{D}$:
\[
	F \boxtimes G 
	:= 
	(\id{C} \otimes \sigma_{C^\ast,D} \otimes \id{D^\ast})
	\circ
	(F \otimes G)
	\circ
	(\id{A} \otimes \sigma_{A^\ast,B}^{-1} \otimes \id{B^\ast})
\]
This makes $(\trace{A})_{A \in \obj{\CategoryC}}$ an \emph{environment structure} for $\CPMFunctor{\CategoryC}$. Furthermore, the specific choice $\trace{A} := \varepsilon_{A}$
satisfies the following additional requirement, which makes $\CPMFunctor{\CategoryC}$ a dagger-compact category:
\[
	(\trace{A})^\dagger = (\trace{A^\ast} \boxtimes \DoublingFunctor{\id{A}}) \circ \DoublingFunctor{\eta_{A}}
\]
The condition above is in fact equivalent to requiring closure of the family under conjugation:
\[
	\trace{A}^\ast = \trace{A^\ast}
\]
In this work, we will keep the additional requirement above explicit, and not include it as part of the definition of environment structure\footnote{This follows the convention set in \cite{gogioso2017categorical} for environment structures in the context of Categorical Probabilistic Theories, where dagger-compact structure is not necessarily of interest.}.

\subsection{A symmetry perspective}
Previous literature on the CPM construction has focussed mostly on the connection with positive operators and completely positive maps \cite{coecke2016pictures,selinger2007dagger} in presence of compact closed structure. States $\rho:I \rightarrow A \otimes A^\ast$ in $\CPMFunctor{\CategoryC}$ correspond to positive morphisms in $\CategoryC$:
\[
	f \circ f^\dagger : A \rightarrow A
\]
where $f: C \rightarrow A$ is some morphism and $C$ is some object. Morphisms in $\CPMFunctor{\CategoryC}$ then correspond to \emph{super-operators} in $\CategoryC$, sending positive morphisms to positive morphisms:
\[
	\Big( f \circ f^\dagger: A \rightarrow A \Big) 
	\mapsto
	\Big( M \circ \big((f \circ f^\dagger) \otimes \id{E} \big) \circ M^\dagger: B \rightarrow B \Big)
\]
where $M: A \otimes E \rightarrow B$ is some morphism and $E$ is some object.

In this work, we will instead take a ``symmetry'' perspective on the CPM construction. To begin with, we observe that the following defines a group homomorphism $\Phi$ from the finite abelian group $\integersMod{2}$ to the group of monoidal automorphisms $\Automs{}{\CategoryC}$ (i.e. monoidal endofunctors with monoidal bilateral inverse) of the category $\CategoryC$:
\[
\begin{array}{cccc}
	\Phi: & \integersMod{2} & \rightarrow & \Automs{}{\CategoryC} \\
	& 0 & \mapsto & \id{\CategoryC} \\
	& 1 & \mapsto & \conjFunctor{\CategoryC}
\end{array}
\]
where $\id{\CategoryC}$ is the identity functor and $\conjFunctor{\CategoryC}$ is the conjugation functor. The doubling functor can then be re-cast as follows, in terms of the group homomorphism $\Phi$:
\[
	\DoublingFunctor{\emptyArg} 
	= 
	\bigotimes\limits_{g \in \integersMod{2}} \Phi(g)[\emptyArg]
	=
	\begin{cases}
		A \mapsto \bigotimes\limits_{g \in \integersMod{2}} \Phi(g)[A] = A \otimes A^\ast\\
		f \mapsto \bigotimes\limits_{g \in \integersMod{2}} \Phi(g)[f] = f \otimes f^\ast\\
	\end{cases}
\]
The actual choice of ordering for the tensor product is essentially irrelevant, as any two choices will lead to functors which are naturally isomorphic via conjugation by a permutation of the objects.

Out of all these natural permutation isomorphisms, we will in particular be interested in the ones corresponding to the regular action of $\integersMod{2}$ on the indices, i.e the natural transformations $\tau(g): \DoublingFunctor{\emptyArg} \Rightarrow \Phi(g)\big[\DoublingFunctor{\emptyArg}\big]$ for all $g \in \integersMod{2}$:
\[
\begin{array}{rcl}
	\tau_{A}(0) & := & \id{A \otimes A^\ast}\\
	\tau_{A}(1) & := & \sigma_{A,A^\ast}
\end{array}
\]
Using them, we can see that the autofunctors $\Phi(g)$ on $\DoubledCategory{\CategoryC}$ are all naturally isomorphic, via conjugation by the permutations $\tau(g)$, to the identity functor:
\[
\begin{array}{ccccc}
	\tau_{B}(0)^{-1} \circ \Phi(0)\big[f \otimes f^\ast\big] \circ \tau_{A}(0)
	& = & 
	\id{B \otimes B^\ast} \circ (f \otimes f^\ast) \circ \id{A \otimes A^\ast}
	& = & 
	f \otimes f^\ast
	\\
	\tau_{B}(1)^{-1} \circ \Phi(1)\big[f \otimes f^\ast\big] \circ \tau_{A}(1)
	& = & 
	\sigma_{B^\ast,B} \circ (f^\ast \otimes f) \circ \sigma_{A,A^\ast}
	& = & 
	f \otimes f^\ast
\end{array}
\]
This means that the morphisms in $\DoubledCategory{\CategoryC}$ are essentially invariant under the $\integersMod{2}$-action given by the autofunctors $\Phi(g)$.

In order to extend this invariance to the morphisms in $\CPMFunctor{\CategoryC}$, we need to make the following assumption on the discarding maps:
\[
	\trace{A}^\ast = \trace{A} \circ \sigma_{A,A^\ast}^{-1}
\]
Note that this assumption is different from the additional assumption $\trace{A}^\ast = \trace{A^\ast}$, but is equally satisfied by the traditional choice $\trace{A} := \varepsilon_{A}$. In terms of the autofunctors and natural transformation above, this means that:
\[
\begin{array}{ccccccc}
	\Phi(0)[\trace{A}] &=& \trace{A} &=& \trace{A} \circ \id{A \otimes A^\ast} &=& \trace{A} \circ \tau_{A}(0)^{-1}\\ 
	\Phi(1)[\trace{A}] &=& \trace{A}^\ast &=& \trace{A} \circ \sigma_{A,A^\ast}^{-1} &=& \trace{A} \circ \tau_{A}(1)^{-1}
\end{array}
\]
As a consequence, we get that the autofunctors $\Phi(g)$ on $\CPMFunctor{\CategoryC}$ are all naturally isomorphic, via conjugation by the permutations $\tau(g)$, to the identity functor (using the fact that the generic morphism in $\CPMFunctor{\CategoryC}$ takes the form $(\id{B} \otimes \trace{E} \otimes \id{B^\ast})\circ(f \otimes f^\ast)$ for some $f:A\rightarrow B \otimes E$ in $\CategoryC$):
\[
\begin{array}{rcl}
	\tau_{B}(0)^{-1}
	\circ
	\Phi(0)\big[
		(\id{B} \otimes \trace{E} \otimes \id{B^\ast})
		\circ
		(f \otimes f^\ast)
	\big]
	\circ
	\tau_{A}(0)
	&=&
		(\id{B} \otimes (\trace{E} \circ \id{E \otimes E^\ast})\otimes \id{B^\ast})
		\circ
		(f \otimes f^\ast)
	\\
	&=&
		(\id{B} \otimes \trace{E}\otimes \id{B^\ast})
		\circ
		(f \otimes f^\ast)
	\\
	\tau_{B}(1)^{-1}
	\circ
	\Phi(1)\big[
		(\id{B} \otimes \trace{E} \otimes \id{B^\ast})
		\circ
		(f \otimes f^\ast)
	\big]
	\circ
	\tau_{A}(1)
	&=&
		(\id{B} \otimes (\trace{E}^\ast \circ \sigma_{E, E^\ast})\otimes \id{B^\ast})
		\circ
		(f \otimes f^\ast)
	\\
	&=&
		(\id{B} \otimes \trace{E}\otimes \id{B^\ast})
		\circ
		(f \otimes f^\ast)
	\\
\end{array}
\]

Restricted to the monoid $(K,\otimes,1)$ of scalars for $\CategoryC$, the autofunctors $\Phi(g)$ define a group homomorphism from $\integersMod{2}$ to the group of monoid automorphisms of $K$. The invariance argument above shows that the scalars of $\DoubledCategory{\CategoryC}$ and $\CPMFunctor{\CategoryC}$ always fall within the sub-set of elements of $K$ which are left fixed by the $\integersMod{2}$-action. Particularly interesting is the case where $\CategoryC$ is enriched in commutative monoids
	\footnote{I.e. homsets come equipped with an addition $+$ and a zero morphism $0$, satisfying appropriate compatibility conditions (e.g. tensor, dagger, associator and unitors must all be linear).}
as a dagger symmetric monoidal category, so that the conjugation functor is automatically linear: this means that $K$ is naturally a commutative semiring $(K,+,0,\otimes,1)$ with involution $^\ast$, and the action of $\integersMod{2}$ coincides with the action of the involution. If $K$ is a field and conjugation is non-trivial, we can define $R$ to be the sub-field fixed by conjugation, and the doubling functor on scalars coincides with the field norm for the quadratic Galois extension $K/R$:
\[
	\DoublingFunctor{x} = x \otimes x^\ast = N_{K/R}(x)
\]
For example, in the case of $\CategoryC = \fHilbCategory$ we have $K=\complexs$ and $^\ast$ is complex conjugation, so that $K/R$ is the quadratic extension $\complexs/\reals$. If $\trace{A}=\sum_{i=1}^n \DoublingFunctor{\bra{a_n}}$ is any test made of product effects and $\DoublingFunctor{\ket{\psi}}$ is a product state on $A \otimes A^\ast$ which is normalised 
	\footnote{I.e. $\trace{A} \circ \DoublingFunctor{\ket{\psi}} = 1$.}
, then the above yields the Born rule for the ($R$-valued) probabilities of test outcomes:
\[
	\mathbb{P}(A_n | \psi) := \DoublingFunctor{\bra{a_n}} \circ \DoublingFunctor{\ket{\psi}} = N_{K/R} (\braket{a_n}{\psi})
\]
In the $\fHilbCategory$ case of $K/R = \complexs / \reals$ we recover the familiar form $N_{K/R}(\braket{a_n}{\psi}) = |\braket{a_n}{\psi}|^2$.

\section{The higher-order CPM construction}

\subsection{The folded category}
Consider a symmetric monoidal category $(\CategoryC,\otimes,I)$. Let $G$ be a finite abelian group, and $\Phi$ be a group homomorphism from $G$ to the group of monoidal automorphisms $\Automs{}{\CategoryC}$. The following definition generalises the construction of the doubled category from the symmetry perspective.
\begin{definition}
The \emph{$\Phi$-folding functor} is the endofunctor $\PhiFoldingFunctorSym{\Phi}$ on $\CategoryC$ defined as follows:
\[
	\PhiFoldingFunctor{\Phi}{\emptyArg}
	=
	\bigotimes\limits_{\gamma \in G} \Phi(\gamma)[\emptyArg]
	=
	\begin{cases}
		A \mapsto \bigotimes\limits_{\gamma \in G} \Phi(\gamma)[A]\\
		f \mapsto \bigotimes\limits_{\gamma \in G} \Phi(\gamma)[f]
	\end{cases}
\]
\end{definition}
\noindent From now on, we will require that $\CategoryC$ be chosen in such a way that the folding functor $\PhiFoldingFunctorSym{\Phi}$ is injective on objects, i.e. that $\PhiFoldingFunctor{\Phi}{A} = \PhiFoldingFunctor{\Phi}{B}$ implies $A = B$ for all $A, B \in \obj{\CategoryC}$. The following is then well-defined.
\begin{definition}
The \emph{$\Phi$-folded category} $\PhiFoldedCategory{\Phi}{\CategoryC}$ is the image of the $\Phi$-folding functor. The folding functor is also a functor $\CategoryC \rightarrow \PhiFoldedCategory{\Phi}{\CategoryC}$ which is bijective on objects.
\end{definition}

\newcounter{lemma_phifolded_c}
\setcounter{lemma_phifolded_c}{\value{theorem_c}}
\begin{lemma}
The $\Phi$-folded category is a symmetric monoidal category $(\PhiFoldedCategory{\Phi}{\CategoryC},\boxtimes,\PhiFoldingFunctor{\Phi}{I})$, with tensor product $\boxtimes$ defined as follows:
\[
	\PhiFoldingFunctor{\Phi}{A} \boxtimes \PhiFoldingFunctor{\Phi}{B}
	 := 
	\PhiFoldingFunctor{\Phi}{A \otimes B}
	\hspace{3cm}
	\PhiFoldingFunctor{\Phi}{f} \boxtimes \PhiFoldingFunctor{\Phi}{g}
	 := 
	\PhiFoldingFunctor{\Phi}{f \otimes g}
\]
The folding functor $\CategoryC \rightarrow \PhiFoldedCategory{\Phi}{\CategoryC}$ is a monoidal functor under this choice of monoidal structure.
\end{lemma}

The choice of ordering for the tensor product is essentially irrelevant---as was the case for the ordinary doubling construction---since all possible choices lead to folding functors which are naturally isomorphic via conjugation by permutations on objects. Once again, we are interested in the natural isomorphisms arising from the regular action of G on the indices of the tensor product, i.e. the $\tau(\gamma) : \PhiFoldingFunctorSym{\Phi} \Rightarrow \Phi(\gamma) \circ \PhiFoldingFunctorSym{\Phi}$ defined as follows for all $\gamma \in G$:
\[
	\tau_{A}(\gamma) 
	:= 
	\text{ the unique permutation }
	\bigotimes\limits_{\delta \in G} \Phi(\delta)[A]
	\longrightarrow
	\bigotimes\limits_{\delta \in G} \Phi(\gamma \delta)[A]
\]
For example, for $G=\integersMod{3}$ we would have the following natural isomorphisms:
\[
\begin{array}{rcl}
	\tau_{A}(0) 
		&:=& 
		\id{A \otimes \Phi(1)[A] \otimes \Phi(2)[A]} 
		\\
	\tau_{A}(1) 
		&:=& 
		(\id{\Phi(1)[A]} \otimes \sigma_{A,\Phi(2)[A]})
		\circ 
		(\sigma_{A,\Phi(1)[A]} \otimes \id{\Phi(2)[A]})
		\\
	\tau_{A}(2) 
		&:=& 
		(\sigma_{A,\Phi(2)[A]} \otimes \id{\Phi(1)[A]})
		\circ 
		(\id{A} \otimes \sigma_{\Phi(1)[A],\Phi(2)[A]})
\end{array}
\]

By using the natural isomorphisms $\tau(\gamma)$, we can see that the monoidal autofunctors $\Phi(\gamma)$ on $\PhiFoldedCategory{\Phi}{\CategoryC}$ are all naturally isomorphic, via conjugation by the permutations $\tau(\gamma)$, to the identity functor:
\[
	\tau_{B}(\gamma)^{-1}
	\circ
	\Phi(\gamma)\Bigg[ \bigotimes\limits_{\delta \in G} \Phi(\delta)[f] \Bigg]
	\circ
	\tau_{A}(\gamma)
	=
	\tau_{B}(\gamma)^{-1}
	\circ
	\bigotimes\limits_{\delta \in G} \Phi(\gamma\delta)[f]
	\circ
	\tau_{A}(\gamma)
	=
	\bigotimes\limits_{\delta \in G} \Phi(\delta)[f]
\]
This means that the morphisms in $\PhiFoldedCategory{\Phi}{\CategoryC}$ are essentially invariant under the $G$-action given by the monoidal autofunctors $\Phi(\gamma)$. If $(K,\otimes,1)$ is the commutative monoid of scalars for $\CategoryC$, then $\Phi$ restricts to an action of $G$ on $K$ by monoid isomorphisms, and we can consider the sub-monoid $R_{\Phi}$ formed by the $G$-invariant elements of $K$. The remarks above show that the scalars of the $\Phi$-folded category are always a sub-monoid of $R_{\Phi}$ (since $\boxtimes$ and $\otimes$ are both the same monoid operation on scalars). If $\CategoryC$ is enriched in commutative monoids, the $\Phi(\gamma)$ autofunctors are linear and $K$ is a field, then the action of the folding functor on scalars corresponds the field norm $N_{K/R_{\Phi}}(x) = \otimes_{\gamma \in G} \Phi(\gamma)[x]$ for the Galois extension $K/R_{\Phi}$.

\subsection{The higher-order CPM construction}

In the previous section, we have seen that an environment structure for the CPM category can be defined by choosing a family of effects---the discarding maps---which respect the monoidal structure of the doubled category. Here we will be interested in the more general context where we choose multiple, compatible environment structures, which we use simultaneously to construct our generalised CPM categories. 

\newpage
Before moving on to do so, recall that the tensor product of maps in the CPM category involved a permutation in order to obtain a domain in the correct form. For the traditional second-order case, this only involved a swap. To deal succinctly with the more general case, we define the following natural isomorphism $\pi = (\pi_{A,B})_{A,B \in \obj{\CategoryC}}$:
\[
	\pi_{A,B}
	:= 
	\text{ the unique permutation }
	\bigotimes\limits_{\gamma \in G} \Phi(\gamma)[A]
	\otimes
	\bigotimes\limits_{\gamma \in G} \Phi(\gamma)[B]
	\longrightarrow
	\bigotimes\limits_{\gamma \in G} \Phi(\gamma)[A \otimes B]
\]
For example, for $G=\integersMod{3}$ we would have the following isomorphism:
\[
	\pi_{A,B}
	=
	(
		\id{A} 
		\otimes 
		\sigma_{\Phi(1)[A],B} 
		\otimes 
		\sigma_{\Phi(2)[A],\Phi(1)[B]}
		\otimes
		\id{\Phi(2)[B]}
	)
	\circ
	(
		\id{A}
		\otimes
		\id{\Phi(1)[A]}
		\otimes
		\sigma_{\Phi(2)[A],B}
		\otimes
		\id{\Phi(1)[B]}
		\otimes
		\id{\Phi(2)[B]}
	)
\]
In particular, the symmetric monoidal structure of the $\Phi$-folded category can be expressed in terms of $\pi$. If $f:A \rightarrow C$ and $g:B \rightarrow D$ are morphisms in $\CategoryC$, then we have that:
\[
	\PhiFoldingFunctor{\Phi}{f} \boxtimes \PhiFoldingFunctor{\Phi}{g}
	=
	\pi_{C,D} 
	\circ 
	(\PhiFoldingFunctor{\Phi}{f} \otimes \PhiFoldingFunctor{\Phi}{g}) 
	\circ
	\pi_{A,B}^{-1} 
\]

\begin{definition}
\label{defi_multiEnvStructure}
A multi-environment structure for $\Phi$ is a family $(\Xi_A)_{A \in \obj{\CategoryC}}$ of sets $\Xi_A$ of effects on $\PhiFoldingFunctor{\Phi}{A}$ in $\CategoryC$ which satisfies the following three conditions:
\begin{enumerate}
	\item[(i)] for all $\xi_A \in \Xi_A$ and all $\xi_B \in \Xi_B$ we have that $(\xi_A \otimes \xi_B) \circ \pi_{A,B}^{-1} \in \Xi_{A \otimes B}$;
	\item[(ii)] we have that $\Xi_{I} = \{1\}$;
	\item[(iii)] for all $\xi_A \in \Xi_A$ and all $\gamma \in G$ we have that $\Phi(\gamma)[\xi_A] = \xi_A \circ \tau_{A}(\gamma)^{-1}$.
\end{enumerate}
In particular, an environment structure for $\Phi$ is a multi-environment structure where each set $\Xi_A$ contains exactly one element, which we denote by $\trace{A}$.
\end{definition}

The multi-environment structures for a fixed $\Phi$ can be partially ordered by object-wise subset inclusion. The partial order is in fact a lattice, with meet given by object-wise set intersection and join given by suitable closure of object-wise set union.

\begin{definition}
\label{defi_PhiXiCPMCategory}
Given a multi-environment structure $\Xi = (\Xi_A)_{A \in \obj{\CategoryC}}$, the \emph{$(\Phi,\Xi)$-CPM category}, which we denote by $\PhiXiCPMCategory{\Phi}{\Xi}{\CategoryC}$, is defined to be the smallest sub-category of $\CategoryC$ which contains $\PhiFoldedCategory{\Phi}{\CategoryC}$ as well as all maps in the following form:
\[
	(\xi_A \otimes \id{\PhiFoldingFunctor{\Phi}{B}}) \circ \pi_{A,B}^{-1}
\]
for all pairs of objects $A,B \in \obj{\CategoryC}$ and all effects $\xi_A \in \Xi_A$ in the multi-environment structure.
\end{definition}

\newcounter{lemma_phixicpm_c}
\setcounter{lemma_phixicpm_c}{\value{theorem_c}}
\begin{lemma}
We can extend the tensor product $\boxtimes$ of $\PhiFoldedCategory{\Phi}{\CategoryC}$ as follows to turn $\PhiXiCPMCategory{\Phi}{\Xi}{\CategoryC}$ into a symmetric monoidal category, having $\PhiFoldedCategory{\Phi}{\CategoryC}$ as a monoidal subcategory:
\[
	F \boxtimes G := \pi_{C,D} \circ (F \otimes G) \circ \pi_{A,B}^{-1}
\]
where $F:\PhiFoldingFunctor{\Phi}{A} \rightarrow \PhiFoldingFunctor{\Phi}{C}$ and $G: \PhiFoldingFunctor{\Phi}{B} \rightarrow \PhiFoldingFunctor{\Phi}{D}$ are generic morphisms in $\PhiXiCPMCategory{\Phi}{\Xi}{\CategoryC}$.
\end{lemma}

We refer to the SMC constructed above as a \emph{higher-order CPM construction}. By analogy to the traditional CPM construction, it is easy to see that the morphisms of $\PhiXiCPMCategory{\Phi}{\Xi}{\CategoryC}$ can always be put into the following normal form---by sliding the multi-environment effects around, and using conditions (i) and (ii) of Definition \ref{defi_multiEnvStructure}---where $f:A \rightarrow B \otimes E$ is a morphism in $\CategoryC$ and $\xi_E \in \Xi_E$ for some $E \in \obj{\CategoryC}$:
\[
	(\id{\PhiFoldingFunctor{\Phi}{B}} \boxtimes \xi_E ) \circ \PhiFoldingFunctor{\Phi}{f}
\]
As a consequence of condition (iii) for the multi-environment structure $\Xi$, we get that the autofunctors $\Phi(\gamma)$ on $\PhiXiCPMCategory{\Phi}{\Xi}{\CategoryC}
$ are all naturally isomorphic, via conjugation by the permutations $\tau(\gamma)$, to the identity functor:
\[
\begin{array}{rcl}
	\tau_{B}(\gamma)^{-1}
	\circ
	\Phi(\gamma)
	\big[
		(\id{\PhiFoldingFunctor{\Phi}{B}} \boxtimes \xi_E)
		\circ
		\PhiFoldingFunctor{\Phi}{f}
	\big]
	\circ
	\tau_{A}(\gamma)
	&=&
	\Big(
		\id{\PhiFoldingFunctor{\Phi}{B}}
		\boxtimes
		\big(\Phi(\gamma)[\xi_E]\circ\tau_{E}(\gamma)\big)
	\Big)
	\circ 
	\PhiFoldingFunctor{\Phi}{f}
	\\
	&=&
	(\id{\PhiFoldingFunctor{\Phi}{B}} \boxtimes \xi_E)
	\circ
	\PhiFoldingFunctor{\Phi}{f}
\end{array}
\]

\subsection{Functoriality of the higher-order CPM construction}

In order to understand the functorial and iterative properties of the higher-order CPM construction, we introduce the notion of a ``universe of symmetric monoidal structures''.

\begin{definition}
Let $\SymMonCatUniverse$ be  the category of (suitably small) symmetric monoidal categories and monoidal functors between them. By a \emph{SMC-universe} we mean a category $\Theta$ equipped with a faithful functor $\UnderlyingSMCFunctor{\emptyArg}{\Theta}:\Theta \rightarrow \SymMonCatUniverse$ with image $\UnderlyingSMC{\Theta}$ forming a sub-category of $\SymMonCatUniverse$. We refer to $\UnderlyingSMCFunctor{\emptyArg}{\Theta}$ as the \emph{underlying SMC functor}.
\end{definition}

\noindent SMC-universes are essentially categories where the objects are symmetric monoidal categories and the morphisms are monoidal functors between them, where the categories can be thought to have been equipped with some additional information, and the morphisms between them restricted based upon that information. Some interesting examples of SMC-universes include the following:
\begin{itemize}
	\item the category $\SymMonCatUniverse$ of symmetric monoidal categories and monoidal functors (equipped with the identity $\UnderlyingSMCFunctor{\emptyArg}{\SymMonCatUniverse} = \id{\SymMonCatUniverse}$);
	\item the category $\DagSymMonCatUniverse$ of dagger symmetric monoidal categories and dagger monoidal functors (equipped with the sub-category inclusion $\UnderlyingSMCFunctor{\emptyArg}{\DagSymMonCatUniverse}$ into ${\SymMonCatUniverse}$);
	\item the category $\DagCompCatUniverse$ of dagger-compact categories with chosen duals and dagger monoidal functors preserving chosen duals (equipped again with the sub-category inclusion into ${\SymMonCatUniverse}$).
\end{itemize}
The reason to introduce the additional layer of abstraction given by the underlying SMC functor, rather than simply considering sub-categories of $\SymMonCatUniverse$ as in the examples above, can be summarised as follows: in order to characterise the higher-order CPM construction as a functor, we need to equip symmetric monoidal categories with additional data (the action $\Phi$ and the multi-environment structure $\Xi$), and we need to restrict the functors allowed between them based on that data (i.e. we will require that our functors be $G$-equivariant and respect the multi-environment structure).

\begin{definition}
By a \emph{morphism of SMC-universes} $(\zeta,n^\zeta):\Theta \rightarrow \Theta'$ we mean a functor $\zeta:\Theta \rightarrow \Theta'$ together with a natural transformation $n^\zeta: \UnderlyingSMCFunctor{\emptyArg}{\Theta} \Rightarrow \UnderlyingSMCFunctor{\zeta(\emptyArg)}{\Theta'}$. We denote the category of SMC-universes and morphisms between them by $\SMCUnivsCategory$.
\end{definition}

\noindent One interesting example of morphism between SMC universes is given by the traditional CPM construction, which we can write as $\CPMFunctor{\emptyArg}:\DagCompCatUniverse\rightarrow\DagCompCatUniverse$ in the following way:
\begin{itemize}
	\item the functor $\zeta: \DagCompCatUniverse \rightarrow \DagCompCatUniverse$ is $\zeta = \CPMFunctor{\emptyArg}$, sending a dagger-compact category $\CategoryC$ to $\CPMFunctor{\CategoryC}$ and a dagger monoidal functor $F:\CategoryC \rightarrow \CategoryD$ to its restriction $\CPMFunctor{\CategoryC} \rightarrow \CPMFunctor{\CategoryD}$ (recalling that $\CPMFunctor{\CategoryC}$ is defined as a sub-category of $\CategoryC$ in this work); 
	\item the natural transformation $\UnderlyingSMCFunctor{\emptyArg}{\DagCompCatUniverse} \Rightarrow \UnderlyingSMCFunctor{\CPMFunctor{\emptyArg}}{\DagCompCatUniverse}$ is given by the doubling functor $\DoublingFunctor{\emptyArg} : \CategoryC \rightarrow \CPMFunctor{\CategoryC}$, where we have used the fact that $\DagCompCatUniverse$ is a sub-category of $\SymMonCatUniverse$, so that we have $\CategoryC = \UnderlyingSMCFunctor{\CategoryC}{\DagCompCatUniverse}$ and $\CPMFunctor{\CategoryC} = \UnderlyingSMCFunctor{\CPMFunctor{\CategoryC}}{\DagCompCatUniverse}$.
\end{itemize}

\begin{definition}
Let $\Theta$ be a SMC-universe. Then the category $\PreCPMFunctor{\Theta}$ is defined as follows.
\begin{enumerate}
	\item[(i)] The objects of $\PreCPMFunctor{\Theta}$ are in the form $(\CategoryC,\Phi,\Xi)$, where;
	\begin{itemize}
		\item $\CategoryC$ is an object in $\Theta$;
		\item $\Phi$ is a group  homomorphism from a finite abelian group $G$ to the automorphisms $\Automs{\Theta}{\CategoryC}$;
		\item $\Xi$ is a multi-environment structure for $\UnderlyingSMCFunctor{\Phi}{\Theta}$;
		\item the $(\UnderlyingSMCFunctor{\Phi}{\Theta},\Xi)$-CPM category $\PhiXiCPMCategory{\UnderlyingSMCFunctor{\Phi}{\Theta}}{\Xi}{\UnderlyingSMCFunctor{\CategoryC}{\Theta}}$ is an object of $\UnderlyingSMC{\Theta}$;
		\item the $\UnderlyingSMCFunctor{\Phi}{\Theta}$-folding functor $\PhiFoldingFunctorSym{\UnderlyingSMCFunctor{\Phi}{\Theta}}: \UnderlyingSMCFunctor{\CategoryC}{\Theta} \rightarrow \PhiXiCPMCategory{\UnderlyingSMCFunctor{\Phi}{\Theta}}{\Xi}{\UnderlyingSMCFunctor{\CategoryC}{\Theta}}$ is a morphism of $\UnderlyingSMC{\Theta}$;
		\item the $\UnderlyingSMCFunctor{\Phi}{\Theta}$-folding functor $\PhiFoldingFunctorSym{\UnderlyingSMCFunctor{\Phi}{\Theta}}$ is injective on objects;
	\end{itemize}
	\item[(ii)] The morphisms $(\CategoryC,\Phi,\Xi) \rightarrow (\CategoryC',\Phi',\Xi')$ in $\PreCPMFunctor{\Theta}$, where $\Phi$ and $\Phi'$ are both actions for the same finite abelian group $G$, are the morphisms $F: \CategoryC \rightarrow \CategoryC'$ in $\Theta$ which satisfy the following:
	\begin{itemize}
		\item $F$ is \emph{$G$-equivariant}, in the sense that for all $\gamma \in G$ we have:
		\[
			F \circ \Phi(\gamma) = \Phi'(\gamma) \circ F
		\]
		\item $\UnderlyingSMCFunctor{F}{\Theta}$ \emph{respects the multi-environment structure}, in the sense that for all $A \in \obj{\UnderlyingSMCFunctor{\CategoryC}{\Theta}}$ we have:
		\[
			\Big\{ \UnderlyingSMCFunctor{F}{\Theta}(\xi_A) \Big| \xi_A \in \Xi_A \Big\} \subseteq \Xi'_{\UnderlyingSMCFunctor{F}{\Theta}(A)}
		\]
	\end{itemize}
	If $\Phi$ and $\Phi'$ are not actions for the same finite abelian group $G$, then there are no morphisms between $(\CategoryC,\Phi,\Xi) $ and $(\CategoryC',\Phi',\Xi')$.
\end{enumerate}
\end{definition}

\newcounter{lemma_precpmuniverse_c}
\setcounter{lemma_precpmuniverse_c}{\value{theorem_c}}
\begin{lemma}
There is a faithful and surjective functor $\UnderlyingSMCFunctor{\emptyArg}{\PreCPMFunctor{\Theta}\rightarrow\Theta}:\PreCPMFunctor{\Theta} \rightarrow \Theta$ which sends $(\CategoryC,\Phi,\Xi)$ to $\CategoryC$ and is the identity on morphisms. The category $\PreCPMFunctor{\Theta}$ is an SMC-universe with underlying SMC functor $\UnderlyingSMCFunctor{\emptyArg}{\PreCPMFunctor{\Theta}}$ defined by $\UnderlyingSMCFunctor{\emptyArg}{\PreCPMFunctor{\Theta}}:= \UnderlyingSMCFunctor{\UnderlyingSMCFunctor{\emptyArg}{\PreCPMFunctor{\Theta}\rightarrow \Theta}}{\Theta}$.
\end{lemma}

\noindent The definition of the SMC-universe $\PreCPMFunctor{\Theta}$ allows us to detail the functorial properties of the higher-order CPM construction in full generality.

\newcounter{lemma_higherordercpmfunctor_c}
\setcounter{lemma_higherordercpmfunctor_c}{\value{theorem_c}}
\begin{lemma}
	Let $\Theta$ be a sub-category of $\SymMonCatUniverse$, seen as a SMC-universe where $\UnderlyingSMCFunctor{\emptyArg}{\Theta}$ is the sub-category inclusion (so that $\UnderlyingSMC{\Theta} = \Theta$). Then the higher-order CPM construction can be used to define a morphism of SMC-universes $(\CPMFunctorSym,n^{\CPMFunctorSym}):\PreCPMFunctor{\Theta} \rightarrow \Theta$ as follows:
	\begin{itemize}
	 	\item the functor $\CPMFunctorSym: \PreCPMFunctor{\Theta} \rightarrow \Theta$ is the one sending an object $(\CategoryC,\Phi,\Xi)$ of $\PreCPMFunctor{\Theta}$ to the object $\PhiXiCPMCategory{\Phi}{\Xi}{\CategoryC}$ of $\Theta$, and acting as the identity on morphisms;
	 	\item the natural transformation $n^{\CPMFunctorSym}$ is given by the $\Phi$-folding functor $\PhiFoldingFunctorSym{\Phi}:\CategoryC \rightarrow \PhiXiCPMCategory{\Phi}{\Xi}{\CategoryC}$.
	 \end{itemize} 
\end{lemma}
\noindent In particular, the result above shows that the higher-order CPM construction is functorial over monoidal functors which are $G$-equivariant and respect the multi-environment structure.

\vspace{-2mm}
\subsection{The higher-order CPM construction as an Eilenberg-Moore algebra}

The traditional CPM construction can be iterated, but the combined result of multiple iterations is not itself a CPM construction: this is because traditional CPM construction is defined to be second-order (i.e. it corresponds to a $\integersMod{2}$ symmetry), while its iterations are higher-order (i.e. they correspond to $\integersMod{2}^n$ symmetries). The higher-order CPM construction does not have this restriction, and its behaviour under iteration can be easily understood in terms of a monad. The entire construction is essentially predicated on the following observations.

If we have two actions $\Phi:G \rightarrow \Automs{\Theta}{\CategoryC}$ and $\Phi':G' \rightarrow \Automs{\Theta}{\CategoryC}$ which commute, i.e. which satisfy $\Phi(\gamma) \Phi'(\gamma') = \Phi'(\gamma')\Phi(\gamma)$ for all $(\gamma,\gamma') \in G \times G'$, then we can combine them into a new action $\Phi \odot \Phi': (G \times G') \rightarrow \Automs{\Theta}{\CategoryC}$ as follows:
\[
	(\Phi \odot \Phi')(\gamma,\gamma') := \Phi(\gamma)\Phi'(\gamma')
\]
This way, $\odot$ can be taken to define a commutative monoid operation on the actions of finite abelian groups over a fixed object $\CategoryC$ of a fixed SMC-universe $\Theta$, where the unit is the trivial action $1:\integersMod{1} \rightarrow \Automs{\Theta}{\CategoryC}$ given by the identity automorphism.

If $\Xi$ is a multi-environment structure for $\UnderlyingSMCFunctor{\Phi}{\Theta}$, then we can define a multi-environment structure $\PhiFoldingFunctor{\Phi'}{\Xi}$ for $\UnderlyingSMCFunctor{\Phi \odot \Phi'}{\Theta}$ as $\PhiFoldingFunctor{\Phi'}{\Xi}_{A} := \Big\{ \PhiFoldingFunctor{\Phi'}{\xi_{A}} \Big| \xi_A \in \Xi_A \Big\}$. Similarly, if $\Xi'$ is a multi-environment structure for $\UnderlyingSMCFunctor{\Phi'}{\Theta}$ then we can define a multi-environment structure $\PhiFoldingFunctor{\Phi}{\Xi'}$ for $\UnderlyingSMCFunctor{\Phi \odot \Phi'}{\Theta}$ as $\PhiFoldingFunctor{\Phi}{\Xi'}_{A} := \Big\{ \PhiFoldingFunctor{\Phi}{\xi'_{A}} \Big| \xi'_A \in \Xi'_A \Big\}$. Given both $\Xi$ and $\Xi'$, we use the lattice operations on multi-environment structures to define the product $\Xi \odot \Xi'$ as a multi-environment structure for $\UnderlyingSMCFunctor{\Phi \odot \Phi'}{\Theta}$:
\[
	\Xi \odot \Xi' := \PhiFoldingFunctor{\Phi'}{\Xi} \bigvee \PhiFoldingFunctor{\Phi}{\Xi'}
\]
We also define the \emph{trivial} multi-environment structure $1$ for $\UnderlyingSMCFunctor{1}{\Theta}$ as follows:
\[
	1_A 
	:= 
	\begin{cases}
		\{1\} &\text{ if } A \isom I\\
		\emptyset &\text{ otherwise}
	\end{cases}
\]
This way, $\odot$ can be taken to define a commutative monoid operation on multi-environment structures, compatible with the commutative monoid operation previously defined on the underlying actions.

\newcounter{theorem_cpmalgebra_c}
\setcounter{theorem_cpmalgebra_c}{\value{theorem_c}}
\begin{theorem}
The map $\PreCPMFunctor{\emptyArg}$ can be extended to an endofunctor of $\SMCUnivsCategory$ by defining its action on morphisms $(\zeta,n^{\zeta}):\Theta \rightarrow \Theta'$ of $\SMCUnivsCategory$ as follows:
\begin{itemize}
	\item the functor $\PreCPMFunctor{\Theta} \rightarrow \PreCPMFunctor{\Theta'}$ is given by:
	\[
	\begin{array}{rcl}
		(\CategoryC,\Phi,\Xi) &\mapsto& (\zeta(\CategoryC),\zeta(\Phi),n^{\zeta}(\Xi)) \\
		F &\mapsto& \zeta(F)
	\end{array}
	\]
	where $\zeta(\Phi)$ is the group homomorphism $\gamma \mapsto \zeta(\Phi(\gamma))$, and we define the multi-environment structure $n^{\zeta}(\Xi)_{A}:=\{ n^{\zeta}_{\CategoryC}(\xi_A) | \xi_A \in \Xi_A \}$;
	\item the natural transformation $\UnderlyingSMCFunctor{\CategoryC}{\Theta} \rightarrow \UnderlyingSMCFunctor{\zeta(\CategoryC)}{\Theta'}$ is given by $n^{\zeta}_{\CategoryC}$.
\end{itemize}
The endofunctor $\PreCPMFunctor{\emptyArg}$ is a monad with the following multiplication $\mu_{\Theta}:\PreCPMFunctor{\PreCPMFunctor{\Theta}} \rightarrow \PreCPMFunctor{\Theta}$ and unit $\eta_{\Theta}: \Theta \rightarrow \PreCPMFunctor{\Theta}$:
\begin{itemize}
	\item the functor $\PreCPMFunctor{\PreCPMFunctor{\Theta}} \rightarrow \PreCPMFunctor{\Theta}$ for the multiplication $\mu$ is given by:\[
	\begin{array}{rcl}
		((\CategoryC,\Phi,\Xi),\Phi',\Xi') &\mapsto& (\CategoryC,\Phi \odot \Phi', \Xi \odot \Xi') \\
		F &\mapsto& F
	\end{array}
	\]
	\item the functor $\Theta \rightarrow \PreCPMFunctor{\Theta}$ for the unit $\eta$ is given by:
	\[
	\begin{array}{rcl}
		\CategoryC &\mapsto& (\CategoryC,1,1) \\
		F &\mapsto& F
	\end{array}
	\]
	\item the natural transformations for both the multiplication $\mu$ and the unit $\eta$ are identity functors:
	\[
		\id{\UnderlyingSMCFunctor{\CategoryC}{\Theta}}: \UnderlyingSMCFunctor{\CategoryC}{\Theta} \rightarrow \UnderlyingSMCFunctor{\CategoryC}{\Theta}
	\]
\end{itemize}
If $\Theta$ is a sub-category of $\SymMonCatUniverse$, then $\CPMFunctorSym: \PreCPMFunctor{\Theta} \rightarrow \Theta$ is an Eilenberg-Moore algebra for the monad $\PreCPMFunctor{\emptyArg}$.
\end{theorem}

\newpage
\section{Examples}

\subsection{Iterated CPM construction}

The simplest example of higher-order CPM construction is given by iterating the traditional second-order CPM construction on a dagger-compact category $\CategoryC$. At the first level, this means choosing the following monoidal $\integersMod{2}$ action $\Phi$ on $\CategoryC$:
\[
\begin{array}{cccc}
	\Phi: & \integersMod{2} & \rightarrow & \Automs{}{\CategoryC} \\
	& 0 & \mapsto & \id{\CategoryC} \\
	& 1 & \mapsto & \conjFunctor{\CategoryC}
\end{array}
\]
together with the environment structure $\Xi_A := \{\varepsilon_A\}$ given by the caps. The $n$-th iteration of the second-order construction is captured by the higher-order construction with $\integersMod{2}^n$ action $\bigodot_{j=1}^{n} \Phi$ and associated multi-environment structure $\bigodot_{j=1}^{n} \Xi$. Explicitly, the group action $\bigodot_{j=1}^{n} \Phi$ takes the following form:
\[
	\Big(\bigodot_{j=1}^{n} \Phi\Big)(b_1,...,b_n) 
	= 
	\Phi(b_n)...\Phi(b_1)
	=
	\big(\conjFunctor{\CategoryC}\big)^{b_1 \oplus ... \oplus b_n}
	=
	\begin{cases}
		\id{\CategoryC} &\text{ if } b_1 \oplus ... \oplus b_n = 0 \\
		\conjFunctor{\CategoryC} &\text { if } b_1 \oplus ... \oplus b_n = 1
	\end{cases}
\]
Explicitly, the effects in the multi-environment structure $\bigodot_{j=1}^{n} \Xi$ are generated by the following effects, for all $i=1,...,n$ and all $A \in \obj{\CategoryC}$:
\[
	\varepsilon_{A}^{(i)} 
	:= 
	\PhiFoldingFunctor{\bigodot_{j=i+1}^{n} \Phi}{ 
		\varepsilon_{\PhiFoldingFunctor{\bigodot_{j=1}^{i-1} \Phi}{A}}
	}
\]
In particular, the double-dilation construction of Zwart and Coecke \cite{zwart2017double} arises as the fourth-order CPM construction with $\integersMod{2}\times\integersMod{2}$ group action and effects in the multi-environment structure generated by:
\[
	\varepsilon_{A}^{(1)} 
	:= 
	\varepsilon_{A} \otimes \varepsilon_{A}^\ast 
	= 
	\PhiFoldingFunctor{\Phi}{\varepsilon_{A}}
	\hspace{3cm}
	\varepsilon_{A}^{(2)}
	:=
	\varepsilon_{A \otimes A^\ast}
	=
	\varepsilon_{\PhiFoldingFunctor{\Phi}{A}}
\]

A handy way of visualising the $\integersMod{2}^n$ symmetry of the iterated CPM construction theory is to imagine the objects in the tensor product $\bigotimes_{(b_1,...,b_n) \in \integersMod{2}^n} \Phi(b_1,...,b_n)[A]$ to be arranged on the vertices of an $n$-dimensional hypercube, centred at the origin and aligned with Cartesian axes in $n$-dimensional space. We take the $i$-th generator $g^{(i)}:=(0,...,0,1,0,...,0)$ of $\integersMod{2}^n$ to act on the hypercube as the reflection $r^{(i)}$ about the $(n-1)$-dimensional hyperplane orthogonal to the $i$-th Cartesian axis, sending each $A$ vertex to the corresponding $A^\ast$ vertex under $\Phi(g^{(i)})$. This way, the $i$-th level generating effect $\varepsilon_{A}^{(i)}$ for the multi-environment structure is exactly the one given by caps connecting each $A$ vertex of the hypercube with the $A^\ast$ vertex obtained via the reflection $r^{(i)}$. In the double-dilation case, this gives the following square diagrams for the generating effects:
\[
	\varepsilon_{A}^{(1)}:=
	\hspace{1mm}
	\scalebox{0.7}{
	\begin{tikzpicture}{}
		\node[circle,fill=black,inner sep=2pt] (bl) at (-1,-1) {};
		\node[circle,fill=black,inner sep=2pt] (br) at (+1,-1) {};
		\node[circle,fill=black,inner sep=2pt] (tl) at (-1,+1) {};
		\node[circle,fill=black,inner sep=2pt] (tr) at (+1,+1) {};
		\draw[->-=.65,line width=2pt] (tl.center) to (tr.center);
		\draw[->-=.65,line width=2pt] (br.center) to (bl.center);
		\node (bl) at (-1.75,-1.75) {$A^\ast$};
		\node (br) at (+1.75,-1.75) {$A$};
		\node (tl) at (-1.75,+1.75) {$A$};
		\node (tr) at (+1.75,+1.75) {$A^\ast$};
	\end{tikzpicture}
	}
	\hspace{2cm}
	\varepsilon_{A}^{(2)}:=
	\hspace{1mm}
	\scalebox{0.7}{
	\begin{tikzpicture}{}
		\node[circle,fill=black,inner sep=2pt] (bl) at (-1,-1) {};
		\node[circle,fill=black,inner sep=2pt] (br) at (+1,-1) {};
		\node[circle,fill=black,inner sep=2pt] (tl) at (-1,+1) {};
		\node[circle,fill=black,inner sep=2pt] (tr) at (+1,+1) {};
		\draw[->-=.65,line width=2pt] (tl.center) to (bl.center);
		\draw[->-=.65,line width=2pt] (br.center) to (tr.center);
		\node (bl) at (-1.75,-1.75) {$A^\ast$};
		\node (br) at (+1.75,-1.75) {$A$};
		\node (tl) at (-1.75,+1.75) {$A$};
		\node (tr) at (+1.75,+1.75) {$A^\ast$};
	\end{tikzpicture}
	}
\]

\subsection{Categories of free finite-dimensional modules}

Iteration of the traditional CPM construction only yields higher-order examples corresponding $\integersMod{2}^n$ conjugating symmetries. In order to construct more interesting examples, we focus on a family of dagger-compact categories with much richer structure, namely the categories $\RMatCategory{S}$ of free finite-dimensional modules over a commutative involutive semiring $S$. In \cite{gogioso2017fantastic}, these categories have been shown to capture a number of well-studied toy models of quantum theory, including real quantum theory, hyperbolic quantum theory, modal quantum theories and the category fRel of finite sets and relations.

We define $\RMatCategory{S}$ to have natural numbers $\naturals$ as objects, and the $m$-by-$n$ $S$-valued matrices as morphisms $n \rightarrow m$. The category has tensor product given by the Kronecker product of matrices, $S$-linear structure given by the $S$-linear structure of matrices and each object $n$ comes with a standard orthonormal basis $\ket{i}_{i=1,...,n}$. Point-wise conjugation of matrices is defined in the standard orthonormal basis using the involution of $S$: dagger, cups and caps are then constructed as in $\fHilbCategory$. From now on, we work in the SMC-universe $\Theta$ of SMCs enriched in commutative monoids with linear functors between them.

A fairly standard way of making higher-order CPM constructions on $S$-Mat is to consider a homomorphism $\varphi:G \rightarrow \Automs{}{S}$ from some finite abelian group $G$ into the semiring automorphisms of $S$. The action $\Phi:G \rightarrow \Automs{\Theta}{\RMatCategory{S}}$ can then be defined as the identity $\Phi(\gamma)[n]=n$ on objects and as follows on morphisms:
\[
	\Phi(\gamma)
	\Bigg(
		\sum_{i=1}^{m} \sum_{j=1}^{n} M_{ij} \ket{i} \bra{j}
	\Bigg)
	:=
	\sum_{i=1}^{m} \sum_{j=1}^{n} \varphi(\gamma)\big(M_{ij}\big) \ket{i} \bra{j}
\]
In particular, picking $G:=\integersMod{2}$ and $\varphi(1) := z \mapsto z^\ast$ will yield a $\Phi$-folded category which is isomorphic to the one obtained from the traditional second-order CPM construction.

Because all objects $n>1$ can be uniquely decomposed (up to permutation) as a product of primes, a multi-environment structure $\Xi$ can be defined by taking sets $\Xi_p$ of effects for all primes $p$, and then closing them under tensor product $\boxtimes$. As a special case, an environment structure generalising the one from the traditional CPM construction can be defined by taking $\Xi_n:=\{\trace{n}\}$ to be the singleton containing the following effect:
\[
	\trace{n} := \sum_{j=1}^{n} \, \PhiFoldingFunctor{\Phi}{\bra{j}}
\]
In the $\integersMod{2} \times \integersMod{2}$ case of double-dilation, we might consider replacing the $\varepsilon_{n}^{(1)}$ effects with the $\trace{n}$ effects defined above, and doing so results in the double-mixing construction.

If $\ket{\psi} := \sum_{j=1}^{n} \psi_j \ket{j}$ is any vector on $n$, then we define its \emph{norm} $\norm{\ket{\psi}}$ to be the following higher-order generalisation of the quadratic trace $\Trace{\ket{\psi}\bra{\psi}}$ for pure states in quantum theory:
\[
	\norm{\ket{\psi}} 
	:= 
	\trace{n} \circ \PhiFoldingFunctor{\Phi}{\ket{\psi}} 
	= 
	\sum_{j=1}^n \norm{\psi_j}
	\hspace{1cm}
	\text{where}
	\hspace{1cm}
	\norm{\psi_j}
	:=
	\prod_{\gamma \in G} \varphi(\gamma)(\psi_j)
\]
Normalisation of a pure state $\PhiFoldingFunctor{\Phi}{\ket{\psi}}$ in the higher-order CPM category $\PhiXiCPMCategory{\Phi}{\Xi}{\RMatCategory{S}}$ has nothing to do with inner products, and is instead the same as having coordinates with norms $\norm{\psi_j}$ adding to 1.\footnote{When $S$ is a field, the norm on scalars $\norm{z}=\prod_{\gamma \in G}\varphi(\gamma)(z)$ establishes a direct connection between higher-order CPM constructions and Galois theory. Specifically, if $K$ is the sub-field of $S$ which is fixed by all field automorphisms $\varphi(\gamma)$, then the norm $\norm{z}$ is exactly the field norm for the finite Galois extension $S/K$.}

By using the traces $\trace{n}$, it is not hard to show that the scalars in $\PhiXiCPMCategory{\Phi}{\Xi}{\RMatCategory{S}}$ are exactly the closure under addition of the subset $\{ \norm{x} | x \in S \} \subseteq S$:
\[
	\sum_{j=1}^{n} \norm{x_j} 
	= 
	\trace{n} \circ \PhiFoldingFunctor{\Phi}{\sum_{j=1}^n x_j \ket{j}} 
\]
As a consequence, the scalars of $\PhiXiCPMCategory{\Phi}{\Xi}{\RMatCategory{S}}$ form a sub-semiring $R$ of $S$. The same trick can be used to show that arbitrary morphisms are closed under addition, so that $\PhiXiCPMCategory{\Phi}{\Xi}{\RMatCategory{S}}$ is enriched in $R$-modules.

\newpage
We now show that a categorical $R$-probabilistic theory \cite{gogioso2017categorical} can be constructed inside the Karoubi envelope for $\PhiXiCPMCategory{\Phi}{\Xi}{\RMatCategory{S}}$: this allows one to study the natural interface between ``quantum'' systems in $\PhiXiCPMCategory{\Phi}{\Xi}{\RMatCategory{S}}$ and ``classical'' systems with a notion of non-determinism defined by the semiring $R$. This includes natural definitions of tests and controlled preparation, as well as the possibility of studying non-locality using no-signalling empirical models from the sheaf-theoretic framework of \cite{abramsky2011sheaf}. When $\varphi(1):= z \mapsto z^\ast$, the construction presented here reduces to the one originally detailed in \cite{gogioso2017fantastic}.
\vspace{-0.5mm}

On every object $n$ of $\RMatCategory{S}$, we can consider the copy map $\!\hbox{\input{symbols/ZbwcomultSym.tex}}\!\!_{n}:= \sum_{j=1}^n (\ket{j} \otimes \ket{j})\circ\bra{j}$ for the special commutative $\dagger$-Frobenius algebra $\hbox{\input{symbols/ZbwdotSym.tex}}\!\!_n$ associated with the standard orthonormal basis.  Combining these maps with the environment structure, we can construct idempotent \emph{decoherence maps} as follows:
\[
	\decoh{\hbox{\input{symbols/ZbwdotSym.tex}}\!\!_n}
	:=
	(\id{n} \otimes \trace{n})
	\circ
	\PhiFoldingFunctor{\Phi}{\!\hbox{\input{symbols/ZbwcomultSym.tex}}\!\!_{n}}
	=
	\sum_{j=1}^n \PhiFoldingFunctor{\Phi}{\ket{j}\bra{j}}
\]

\newcounter{lemma_SMatKaroubi_c}
\setcounter{lemma_SMatKaroubi_c}{\value{theorem_c}}
\begin{lemma}
The full subcategory of the Karoubi envelope for $\PhiXiCPMCategory{\Phi}{\Xi}{\RMatCategory{S}}$ spanned by objects in the form $(n,\decoh{\hbox{\input{symbols/ZbwdotSym.tex}}\!\!_n})$ is isomorphic to $\RMatCategory{R}$, i.e. it behaves as the category of $R$-probabilistic classical systems. As a consequence, the full sub-SMC of the Karoubi envelope spanned by objects in the form $(n,\id{n})$---the ``quantum'' systems---and the objects in the form $(n,\decoh{\hbox{\input{symbols/ZbwdotSym.tex}}\!\!_n})$---the ``classical'' systems---defines a categorical $R$-probabilistic theory in the sense of \cite{gogioso2017categorical}.
\end{lemma}

In the categorical $R$-probabilistic theory defined above, a generic normalised quantum-to-classical process $A:(n,\id{n}) \rightarrow (m,\decoh{\hbox{\input{symbols/ZbwdotSym.tex}}\!\!_m})$---the generalisation of a POVM, if you will---is defined by a classically-indexed family $(A_i)_{i=1,...,m}$ of effects $A_i:n\rightarrow 1$ in $\PhiXiCPMCategory{\Phi}{\Xi}{\RMatCategory{S}}$ such that $\sum_{i=1}^m A_i = \trace{n}$. In the sharp case where $A_i = \PhiFoldingFunctor{\Phi}{\bra{a_i}}$ for some orthonormal family $\ket{a_i}$ of states in $n$, the Born rule determining the $R$-valued probability of outcome $i$ on pure state $\ket{\psi}$ takes the following higher-order form:
\[
	\mathbb{P}(i | \psi) 
	=
	A_i \circ \PhiFoldingFunctor{\Phi}{\ket{\psi}}
	=
	\PhiFoldingFunctor{\Phi}{\braket{a_i}{\psi}}
	=
	\norm{\braket{a_i}{\psi}}
\]
The non-quadratic nature of the Born rule suggests that categorical $R$-probabilistic theories obtained using the higher-order CPM construction might display higher-order interference phenomena, and this indeed turns out to be the case: in recent work by \cite{gogioso2018hypercubes}, a variation on the double-mixing construction was used to construct a probabilistic theory of ``density hypercubes'', displaying interference of order up to four and possessing hyper-decoherence maps.

\section{Conclusions and future work}

We have shown that the CPM construction can be generalised from the traditional $\integersMod{2}$ conjugating symmetry to arbitrary finite abelian group symmetries, in a completely functorial way. We have provided a categorical description of the closure of our higher-order CPM constructions under iteration, characterising it via the definition of an Eilenberg-Moore algebra for a suitable monad.

We have constructed a broad family of semiring-based examples generalising the traditional second-order ones, and we have proved that they can all be studied using the operational framework of categorical probabilistic theories. As shown by recent work on higher-order interference and hyper-decoherence, these new examples have real potential to provide a wealth of previously unknown exotic models, and we look forward to studying them in detail as part of future work.

Finally, this work defines generalised CPM constructions in such a way that they are embedded into the original SMC, for mathematical ease of definitions and proofs. This leads to the technical requirement that the folding functor be injective on objects, which could be avoided by adopting a construction analogous to \cite{selinger2007dagger}. Such a modification is conceptually simple but technically convoluted, and is left to future work.

\bibliographystyle{eptcs}
\nocite{*}
\bibliography{bibliography}

\newpage
\appendix 

\section{Proofs}

\setcounter{theorem_c}{\value{lemma_phifolded_c}}
\begin{lemma}
The $\Phi$-folded category is a symmetric monoidal category $(\PhiFoldedCategory{\Phi}{\CategoryC},\boxtimes,\PhiFoldingFunctor{\Phi}{I})$, with tensor product $\boxtimes$ defined as follows:
\[
\begin{array}{rcl}
	\PhiFoldingFunctor{\Phi}{A} \boxtimes \PhiFoldingFunctor{\Phi}{B}
	& := &
	\PhiFoldingFunctor{\Phi}{A \otimes B}
	\\
	\PhiFoldingFunctor{\Phi}{f} \boxtimes \PhiFoldingFunctor{\Phi}{g}
	& := &
	\PhiFoldingFunctor{\Phi}{f \otimes g}
	\\
\end{array}
\]
The folding functor $\CategoryC \rightarrow \PhiFoldedCategory{\Phi}{\CategoryC}$ is a monoidal functor under this choice of monoidal structure.
\end{lemma}
\begin{proof}
The proof is entirely straightforward.
\end{proof}

\setcounter{theorem_c}{\value{lemma_phixicpm_c}}
\begin{lemma}
We can extend the tensor product $\boxtimes$ of $\PhiFoldedCategory{\Phi}{\CategoryC}$ as follows to turn $\PhiXiCPMCategory{\Phi}{\Xi}{\CategoryC}$ into a symmetric monoidal category, having $\PhiFoldedCategory{\Phi}{\CategoryC}$ as a monoidal subcategory:
\[
	F \boxtimes G := \pi_{C,D} \circ (F \otimes G) \circ \pi_{A,B}^{-1}
\]
where $F:\PhiFoldingFunctor{\Phi}{A} \rightarrow \PhiFoldingFunctor{\Phi}{C}$ and $G: \PhiFoldingFunctor{\Phi}{B} \rightarrow \PhiFoldingFunctor{\Phi}{D}$ are generic morphisms in $\PhiXiCPMCategory{\Phi}{\Xi}{\CategoryC}$.
\end{lemma}
\begin{proof}
note that $\PhiXiCPMCategory{\Phi}{\Xi}{\CategoryC}$ is defined by composing all maps from $\PhiFoldedCategory{\Phi}{\CategoryC}$, where $\boxtimes$ is already well-defined, with all maps in the following form, for all effects $\xi_A \in \Xi_A$ in the multi-environment structure:
\[
	(\xi_A \otimes \id{\PhiFoldingFunctor{\Phi}{B}}) \circ \pi_{A,B}^{-1}
\]
Note that by the way we have extended the definition of $\boxtimes$ from $\PhiFoldedCategory{\Phi}{\CategoryC}$ to $\PhiXiCPMCategory{\Phi}{\Xi}{\CategoryC}$, the maps above can be equivalently written in the following form:
\[
	\xi_{A} \boxtimes \id{\PhiFoldingFunctor{\Phi}{B}}
\]
By picking a suitable composite system for $B$ and using the symmetry isomorphisms from $\PhiFoldedCategory{\Phi}{\CategoryC}$, we can place the effect $\xi_A$ on any output of type $\PhiFoldingFunctor{\Phi}{A}$ of any map in $\PhiFoldedCategory{\Phi}{\CategoryC}$. It is therefore enough to show that things work out in the case of $\xi_A \boxtimes \xi_B$, for any choice of $\xi_A \in \Xi_A$ and $\xi_B \in \Xi_B$. But this follows from condition (i) of the definition of multi-environment structure, since the $\boxtimes$ product of two effects in the multi-environment structure is again in the multi-environment structure:
\[
	\xi_A \boxtimes \xi_B = (\xi_A \otimes \xi_B) \circ \pi_{A,B}^{-1} \in \Xi_{A \otimes B}
\] 
\end{proof}
\begin{remark*}
One may wonder why we didn't simply define $\PhiXiCPMCategory{\Phi}{\Xi}{\CategoryC}$ as generated by $\PhiFoldedCategory{\Phi}{\CategoryC}$, $\boxtimes$ and the effects in $\Xi$. The reasons for this is technical: the product $\boxtimes$ is not defined on the entirety of $\CategoryC$, so such a definition would be mathematically imprecise. The reader should however feel free to reason about $\PhiXiCPMCategory{\Phi}{\Xi}{\CategoryC}$ ``as if'' that was its definition. 
\end{remark*}
\begin{remark*}
One may also wonder why we had to consider maps in the form $\xi_{A} \boxtimes \id{\PhiFoldingFunctor{\Phi}{B}}$, rather than going directly for $\xi_A \boxtimes \xi_B$. This is because $\PhiXiCPMCategory{\Phi}{\Xi}{\CategoryC}$ had to be defined---for the reasons explained in the previous Remark---as a sub-\underline{category} of $\CategoryC$ spanned by certain maps, and not as a sub-\underline{SMC} of $\CategoryC$. This means that the existence of maps in the form $\xi_{A} \boxtimes \id{\PhiFoldingFunctor{\Phi}{B}}$ would not automatically follow from the existence of effects in the form $\xi_A \boxtimes \xi_B$.
\end{remark*}

\setcounter{theorem_c}{\value{lemma_precpmuniverse_c}}
\begin{lemma}
There is a faithful and surjective functor $\UnderlyingSMCFunctor{\emptyArg}{\PreCPMFunctor{\Theta}\rightarrow\Theta}:\PreCPMFunctor{\Theta} \rightarrow \Theta$ which sends $(\CategoryC,\Phi,\Xi)$ to $\CategoryC$ and is the identity on morphisms. The category $\PreCPMFunctor{\Theta}$ is an SMC-universe with underlying SMC functor $\UnderlyingSMCFunctor{\emptyArg}{\PreCPMFunctor{\Theta}}$ defined by $\UnderlyingSMCFunctor{\emptyArg}{\PreCPMFunctor{\Theta}}:= \UnderlyingSMCFunctor{\UnderlyingSMCFunctor{\emptyArg}{\PreCPMFunctor{\Theta}\rightarrow \Theta}}{\Theta}$.
\end{lemma}
\begin{proof}
The proof is entirely straightforward.
\end{proof}

\setcounter{theorem_c}{\value{lemma_higherordercpmfunctor_c}}
\begin{lemma}
	Let $\Theta$ be a sub-category of $\SymMonCatUniverse$, seen as a SMC-universe where $\UnderlyingSMCFunctor{\emptyArg}{\Theta}$ is the sub-category inclusion (so that $\UnderlyingSMC{\Theta} = \Theta$). Then the higher-order CPM construction can be used to define a morphism of SMC-universes $(\CPMFunctorSym,n^{\CPMFunctorSym}):\PreCPMFunctor{\Theta} \rightarrow \Theta$ as follows:
	\begin{itemize}
	 	\item the functor $\CPMFunctorSym: \PreCPMFunctor{\Theta} \rightarrow \Theta$ is the one sending an object $(\CategoryC,\Phi,\Xi)$ of $\PreCPMFunctor{\Theta}$ to the object $\PhiXiCPMCategory{\Phi}{\Xi}{\CategoryC}$ of $\Theta$, and acting as the identity on morphisms;
	 	\item the natural transformation $n^{\CPMFunctorSym}$ is given by the $\Phi$-folding functor $\PhiFoldingFunctorSym{\Phi}:\CategoryC \rightarrow \PhiXiCPMCategory{\Phi}{\Xi}{\CategoryC}$.
	 \end{itemize} 
\end{lemma}
\begin{proof}
First we show that the functor $\CPMFunctorSym: \PreCPMFunctor{\Theta} \rightarrow \Theta$ is well-defined. Condition (ii) in the definition of $\PreCPMFunctor{\Theta}$ requires the following things for a morphism $F: (\CategoryC,\Phi,\Xi) \rightarrow (\CategoryC',\Phi',\Xi')$ to be in $\PreCPMFunctor{\Theta}$:
\begin{itemize}
	\item $F: \CategoryC \rightarrow \CategoryC'$ is in $\Theta$, and in particular it is monoidal;
	\item $\Phi$ and $\Phi'$ are actions of the same $G$, and the functor $F$ is $G$-equivariant: $\Phi'(\gamma) \circ F = F \circ \Phi(\gamma)$;
	\item $F$ respects the multi-environment structure: $\{F(\xi_A) | \xi_A \in \Xi_A \} \subseteq \Xi'_{F(A)}$.
\end{itemize}
By monoidality and $G$-equivariance, we must have that $F(\PhiFoldingFunctor{\Phi}{A}) = \PhiFoldingFunctor{\Phi'}{F(A)}$ on objects and $F(\PhiFoldingFunctor{\Phi}{f}) = \PhiFoldingFunctor{\Phi'}{F(f)}$ on morphisms, so $F$ restricts to a well-defined monoidal functor $F:\PhiFoldedCategory{\Phi}{\CategoryC}\rightarrow \PhiFoldedCategory{\Phi'}{\CategoryC'}$. Respect of the multi-environment structure implies that $F$ further extends to a well-defined monoidal functor $F: \PhiXiCPMCategory{\Phi}{\Xi}{\CategoryC}\rightarrow\PhiXiCPMCategory{\Phi'}{\Xi'}{\CategoryC'}$.

We now show that the natural transformation $n^{\CPMFunctorSym}$ is well-defined. We know that $\PhiFoldingFunctor{\Phi}{\CategoryC}: \CategoryC \rightarrow \PhiFoldedCategory{\Phi}{\CategoryC}$ is a morphism in $\Theta$, as part of condition (i) in the definition of $\PreCPMFunctor{\Theta}$, so all we need to show is that the functors $\PhiFoldingFunctor{\Phi}{\CategoryC}: \CategoryC \rightarrow \PhiFoldedCategory{\Phi}{\CategoryC}$ and $\PhiFoldingFunctor{\Phi'}{\CategoryC'}: \CategoryC' \rightarrow \PhiFoldedCategory{\Phi'}{\CategoryC'}$ satisfy:
\[
	\PhiFoldingFunctor{\Phi'}{\CategoryC'} \circ F
	=
	F \circ \PhiFoldingFunctor{\Phi}{\CategoryC}
\]
for every morphism $F:(\CategoryC,\Phi,\Xi) \rightarrow (\CategoryC',\Phi',\Xi')$ in $\PreCPMFunctor{\Theta}$. Note that on the LHS of the equation we are using $F:\CategoryC \rightarrow \CategoryC'$, while on the RHS of the equation we are using $F: \PhiXiCPMCategory{\Phi}{\Xi}{\CategoryC}\rightarrow\PhiXiCPMCategory{\Phi'}{\Xi'}{\CategoryC'}$. The equation above then holds as a consequence of monoidality, $G$-equivariance and respect of multi-environment structure on part of $F$.
\end{proof}

\setcounter{theorem_c}{\value{theorem_cpmalgebra_c}}
\begin{theorem}
The map $\PreCPMFunctor{\emptyArg}$ can be extended to an endofunctor of $\SMCUnivsCategory$ by defining its action on morphisms $(\zeta,n^{\zeta}):\Theta \rightarrow \Theta'$ of $\SMCUnivsCategory$ as follows:
\begin{itemize}
	\item the functor $\PreCPMFunctor{\Theta} \rightarrow \PreCPMFunctor{\Theta'}$ is given by:
	\[
	\begin{array}{rcl}
		(\CategoryC,\Phi,\Xi) &\mapsto& (\zeta(\CategoryC),\zeta(\Phi),n^{\zeta}(\Xi)) \\
		F &\mapsto& \zeta(F)
	\end{array}
	\]
	where $\zeta(\Phi)$ is the group homomorphism $\gamma \mapsto \zeta(\Phi(\gamma))$, and we define the multi-environment structure $n^{\zeta}(\Xi)_{A}:=\{ n^{\zeta}_{\CategoryC}(\xi_A) | \xi_A \in \Xi_A \}$;
	\item the natural transformation $\UnderlyingSMCFunctor{\CategoryC}{\Theta} \rightarrow \UnderlyingSMCFunctor{\zeta(\CategoryC)}{\Theta'}$ is given by $n^{\zeta}_{\CategoryC}$.
\end{itemize}
The endofunctor $\PreCPMFunctor{\emptyArg}$ is a monad with the following multiplication $\mu_{\Theta}:\PreCPMFunctor{\PreCPMFunctor{\Theta}} \rightarrow \PreCPMFunctor{\Theta}$ and unit $\eta_{\Theta}: \Theta \rightarrow \PreCPMFunctor{\Theta}$:
\begin{itemize}
	\item the functor $\PreCPMFunctor{\PreCPMFunctor{\Theta}} \rightarrow \PreCPMFunctor{\Theta}$ for the multiplication $\mu$ is given by:\[
	\begin{array}{rcl}
		((\CategoryC,\Phi,\Xi),\Phi',\Xi') &\mapsto& (\CategoryC,\Phi \odot \Phi', \Xi \odot \Xi') \\
		F &\mapsto& F
	\end{array}
	\]
	\item the functor $\Theta \rightarrow \PreCPMFunctor{\Theta}$ for the unit $\eta$ is given by:
	\[
	\begin{array}{rcl}
		\CategoryC &\mapsto& (\CategoryC,1,1) \\
		F &\mapsto& F
	\end{array}
	\]
	\item the natural transformations for both the multiplication $\mu$ and the unit $\eta$ are identity functors:
	\[
		\id{\UnderlyingSMCFunctor{\CategoryC}{\Theta}}: \UnderlyingSMCFunctor{\CategoryC}{\Theta} \rightarrow \UnderlyingSMCFunctor{\CategoryC}{\Theta}
	\]
\end{itemize}
If $\Theta$ is a sub-category of $\SymMonCatUniverse$, then $\CPMFunctorSym: \PreCPMFunctor{\Theta} \rightarrow \Theta$ is an Eilenberg-Moore algebra for the monad $\PreCPMFunctor{\emptyArg}$.
\end{theorem}
\begin{proof}
There are a number of claims to check: we need to show that the action on morphisms is well-defined, we need to show that the monad laws hold, and we need to show that the algebra laws hold.

Given a morphism $(\zeta,n^\zeta):\Theta \rightarrow \Theta'$ of $\SMCUnivsCategory$, we need to check that the following gives a well-defined functor $\PreCPMFunctor{\Theta} \rightarrow \PreCPMFunctor{\Theta'}$:
\[
\begin{array}{rcl}
	(\CategoryC,\Phi,\Xi) & \mapsto & (\zeta(\CategoryC),\zeta(\Phi),n^\zeta(\Xi)) \\
	F & \mapsto & \zeta(F) 
\end{array}
\] 
Firstly, we need to check that $(\zeta(\CategoryC),\zeta(\Phi),n^\zeta(\Xi))$ is an object of $\PreCPMFunctor{\Theta'}$ whenever $(\CategoryC,\Phi,\Xi)$ is an object of $\PreCPMFunctor{\Theta}$. Because $\zeta: \Theta \rightarrow \Theta'$ is a functor, $\zeta(\CategoryC)$ is necessarily in $\Theta'$ and the morphisms $\zeta(\Phi(\gamma)): \zeta(\CategoryC) \rightarrow \zeta(\CategoryC)$ define an action of $G$ on $\zeta(\CategoryC)$ in $\Theta'$. Monoidality of $n^{\zeta}_{\CategoryC}$ furthermore ensures that conditions (i) and (ii) in the definition of multi-environment structure are satisfied for $n^\zeta(\Xi)$, while condition (iii) follows from naturality:
\[
	\zeta(\Phi(\gamma))[n^\zeta_{\CategoryC}(\xi_A)]
	=
	n^\zeta_{\CategoryC}\big(\Phi(\gamma)[\xi_A]\big)
	=
	n^\zeta_{\CategoryC}\big(\xi_{A} \circ \tau_A(\gamma)^{-1}\big)
	=
	n^\zeta_{\CategoryC}(\xi_A)\circ \tau_{n^\zeta_{\CategoryC}(A)}(\gamma)^{-1}
\]
Secondly, we need to check that $F \mapsto \zeta(F)$ provides a well-defined action of the functor $\PreCPMFunctor{\Theta} \rightarrow \PreCPMFunctor{\Theta'}$ on morphisms $F:(\CategoryC,\Phi,\Xi) \rightarrow (\CategoryC',\Phi',\Xi')$. The $G$-equivariance requirement on $\zeta(F)$ follows from functoriality:
\[
	\zeta(\Phi'(\gamma)) \circ \zeta(F)
	=
	\zeta(\Phi'(\gamma)\circ F)
	=
	\zeta(F \circ \Phi(\gamma))
	=
	\zeta(F) \circ \zeta(\Phi(\gamma))
\]
Respect of the multi-environment structure follows form naturality of $n^\zeta$:
\[
	\UnderlyingSMCFunctor{\zeta(F)}{\Theta'}\big(n^\zeta_{\CategoryC}(\xi_A)\big)
	=
	n^\zeta_{\CategoryC'}\big(\UnderlyingSMCFunctor{F}{\Theta}(\xi_A)\big)
\]
Finally, we need to check that $n^\zeta_{\CategoryC}$ provides a suitable natural transformation. But this is obvious.

Having shown that $\PreCPMFunctor{\emptyArg}$ is an endofunctor, we move on to establishing that axioms for a monad are satisfied by the given multiplication $\mu_{\Theta}:\PreCPMFunctor{\PreCPMFunctor{\Theta}} \rightarrow \PreCPMFunctor{\Theta}$ and unit $\eta_{\Theta}: \Theta \rightarrow \PreCPMFunctor{\Theta}$. In fact, we don't have to do any work here: all necessary commuting diagrams follows from associativity and unitality of the monoids $(\odot,1)$ formed by the group actions and the multi-environment structures. The only thing to check is that the product $\Phi \odot \Phi'$ is well-defined, i.e. that the actions $\Phi$ and $\Phi'$ commute: this is an immediate consequence of the fact that the automorphisms $\Phi'(\gamma):(\CategoryC,\Phi,\Xi) \rightarrow (\CategoryC,\Phi,\Xi)$ live in $\PreCPMFunctor{\Theta}$, and hence are $G$-equivariant with respect to $\Phi$.

The only thing remaining to be shown is that $\CPMFunctorSym: \PreCPMFunctor{\Theta} \rightarrow \Theta$ is an Eilenberg-Moore algebra for the monad $\PreCPMFunctor{\Theta}$, whenever $\Theta$ is a sub-category of $\SymMonCatUniverse$: the actual definition of the monoid operations $(\odot,1)$ on group actions and multi-environment structures comes into play here.

We begin by checking the triangle law for algebras. Recall that the unit for group actions is defined by the trivial action $1:\integersMod{1} \rightarrow \Automs{\Theta}{\CategoryC}$ given by the identity automorphism, and that the unit for multi- environment structures is defined by:
\[
	1_A 
	:= 
	\begin{cases}
		\{1\} &\text{ if } A \isom I\\
		\emptyset &\text{ otherwise}
	\end{cases}
\]
The category $\PhiXiCPMCategory{1}{1}{\CategoryC} = \CPMFunctor{(\CategoryC,1,1)}$ is simply $\CategoryC$, and we have $\CPMFunctor{F} = F$ on functors $F:\CategoryC \rightarrow \CategoryC'$, so the triangle law $\CPMFunctorSym \circ \eta_{\Theta} = \id{\Theta}$ is satisfied as desired.

Recall now that the product $\Phi \odot \Phi': (G \otimes G') \rightarrow \Automs{\Theta}{\CategoryC}$ between commuting actions $\Phi: G \rightarrow \Automs{\Theta}{\CategoryC}$ and $\Phi':G' \rightarrow \Automs{\Theta}{\CategoryC}$ is defined by:
\[
	(\Phi \odot \Phi')(\gamma,\gamma') := \Phi'(\gamma') \Phi(\gamma)
\]
Also recall that the product $\Xi \odot \Xi'$ between multi-environment structures is defined by:
\[
	\Xi \odot \Xi' := \PhiFoldingFunctor{\Phi'}{\Xi} \bigvee \PhiFoldingFunctor{\Phi}{\Xi'}
\]
where we have taken $\PhiFoldingFunctor{\Phi'}{\Xi}_{A} := \Big\{ \PhiFoldingFunctor{\Phi'}{\xi_{A}} \Big| \xi_A \in \Xi_A \Big\} $ and $\PhiFoldingFunctor{\Phi}{\Xi'}_{A} := \Big\{ \PhiFoldingFunctor{\Phi}{\xi'_{A}} \Big| \xi'_A \in \Xi'_A \Big\}$. We make the following observations:
\begin{itemize}
	\item by construction, the product $\Phi \odot \Phi'$ yields the same folding as $\Phi$ followed by $\Phi'$:
	\[
		\bigotimes\limits_{(\gamma,\gamma') \in G \times G'}
		\Phi'(\gamma') \Phi(\gamma) [\emptyArg]
		=
		\bigotimes\limits_{\gamma' \in G'} \Phi'(\gamma')
		\Bigg[
		\bigotimes\limits_{\gamma \in G} \Phi(\gamma)[\emptyArg]
		\Bigg]
	\]
	\item by construction, the product $\Xi \odot \Xi'$ is generated by the following effects:
	\[
		\bigotimes\limits_{\gamma' \in G'} \Phi'(\gamma')[\xi]
		\hspace{3cm}
		\bigotimes\limits_{\gamma \in G} \Phi(\gamma)[\xi']
	\]
	The former are the effects $\xi \in \Xi$ after $\Phi'$-folding, while the latter are the effects $\xi' \in \Xi'$ to which the natural transformation $n^{\CPMFunctorSym}_{\CategoryC}=\PhiFoldingFunctorSym{\Phi}:\CategoryC \rightarrow \PhiXiCPMCategory{\Phi}{\Xi}{\CategoryC}$ has been applied as part of the action of the monad $\PreCPMFunctor{\emptyArg}$
\end{itemize}
This means that the iterated construction
\[
	\PhiXiCPMCategory{\Phi'}{\Xi'}{\PhiXiCPMCategory{\Phi}{\Xi}{\CategoryC}}
	=
	\CPMFunctorSym 
	\circ
	\PreCPMFunctor{\CPMFunctorSym}
	\bigg[
	(\CategoryC, \Phi \odot \Phi', \Xi \odot \Xi')
	\bigg]
\]
results in the same exact sub-category of $\CategoryC$ as the one-shot construction
\[
	\PhiXiCPMCategory{\Phi\odot \Phi'}{\Xi \odot \Xi'}{\CategoryC}
	=
	\CPMFunctorSym \circ \mu_{\Theta}
	\bigg[
		\big((\CategoryC,\Phi,\Xi),\Phi',\Xi'\big)
	\bigg]
\]
Similar considerations can be made on the action of the $\CPMFunctorSym$ construction on morphisms, showing that the square law $\CPMFunctorSym \circ \PreCPMFunctor{\CPMFunctorSym} = \CPMFunctorSym \circ \mu_{\Theta}$ is satisfied, as desired.
\end{proof}

\setcounter{theorem_c}{\value{lemma_SMatKaroubi_c}}
\begin{lemma}
The full subcategory of the Karoubi envelope for $\PhiXiCPMCategory{\Phi}{\Xi}{\RMatCategory{S}}$ spanned by objects in the form $(n,\decoh{\hbox{\input{symbols/ZbwdotSym.tex}}\!\!_n})$ is isomorphic to $\RMatCategory{R}$, i.e. it behaves as the category of $R$-probabilistic classical systems. As a consequence, the full sub-SMC of the Karoubi envelope spanned by objects in the form $(n,\id{n})$---the ``quantum'' systems---and the objects in the form $(n,\decoh{\hbox{\input{symbols/ZbwdotSym.tex}}\!\!_n})$---the ``classical'' systems---defines a categorical $R$-probabilistic theory in the sense of \cite{gogioso2017categorical}.
\end{lemma}
\begin{proof}
The proof is entirely straightforward, analogous to the proof given in \cite{gogioso2017fantastic} for the second-order case of ``conjugation'' in involutive semirings. 
\end{proof}

\end{document}

%% file: packages.tex
\usepackage{mathtools} 
\usepackage{amstext} 
\usepackage{amssymb} 
\usepackage{bbold} 
\usepackage{amsthm} 
\usepackage{stmaryrd} 

\usepackage{relsize} 
\usepackage{microtype} 
\usepackage{csquotes} 
\usepackage{ragged2e} 
\usepackage{cancel}

\usepackage{hyperref} 
\usepackage[nocompress]{cite} 

\usepackage{graphicx} 
\usepackage[usenames,dvipsnames]{xcolor} 
\usepackage{stackengine}
\usepackage{tikz} 
\usepackage{circuitikz} 
\usetikzlibrary{
	arrows,
	shapes,
	decorations,
	intersections,
	backgrounds,
	positioning,
	circuits.ee.IEC
	}


%% file: environments.tex
		\newcounter{theorem_c} 

		\theoremstyle{plain} 
		\newtheorem{theorem}[theorem_c]{Theorem}
		
		\newtheorem{lemma}[theorem_c]{Lemma}

		\newtheoremstyle{exampstyle}
		  {2mm} 
		  {2mm} 
		  {\itshape} 
		  {} 
		  {\bfseries} 
		  {.} 
		  {.5em} 
		  {} 

		\theoremstyle{exampstyle}
		\newtheorem{definition}[theorem_c]{Definition}

		\newtheorem*{remark*}{Remark}

%% file: macros.tex
	\newcommand{\emptyArg}{\,\underline{\hspace{6px}}\,} 

	\newcommand{\naturals}{\mathbb{N}} 
	\newcommand{\reals}{\mathbb{R}} 
	\newcommand{\complexs}{\mathbb{C}} 
	\newcommand{\integersMod}[1]{\mathbb{Z}_{#1}} 
	
		\newcommand{\ket}[1]{\vert #1 \rangle} 
		\newcommand{\bra}[1]{\langle #1 \vert} 
		\newcommand{\braket}[2]{\langle #1 \vert #2 \rangle} 
		
		\newcommand{\Trace}[1]{\operatorname{Tr}\,#1} 
		\newcommand{\decohSym}{\operatorname{dec}} 
		\newcommand{\decoh}[1]{\decohSym_{#1}} 

		\newcommand{\isom}{\cong} 
		\newcommand{\id}[1]{id_{#1}} 

		\newcommand{\Automs}[2]{\operatorname{Aut}_{\,#1}\left[#2\right]} 

		\newcommand{\RMatCategory}[1]{#1\operatorname{-Mat}} 
		\newcommand{\fHilbCategory}{\operatorname{fHilb}} 

		\newcommand{\CategoryC}{\mathcal{C}}
		\newcommand{\CategoryD}{\mathcal{D}}

		\newcommand{\obj}[1]{\operatorname{obj} \, #1} 
		
		\newcommand{\DoubledCategory}[1]{\operatorname{Double}\left[#1\right]}

	\newcommand{\Xbwcolour}{black!80}

	\newcommand{\Zbwcolour}{white}

	\newcommand{\Ybwcolour}{black!15}

	\newcommand{\Wbwcolour}{black!50}

	\newcommand{\trace}[1]{\hbox{\input{symbols/traceSym.tex}}\!_{#1}} 

	\tikzset{
	  rectangle with rounded corners north west/.initial=4pt,
	  rectangle with rounded corners south west/.initial=4pt,
	  rectangle with rounded corners north east/.initial=4pt,
	  rectangle with rounded corners south east/.initial=4pt,
	}
	\makeatletter
	\pgfdeclareshape{rectangle with rounded corners}{
	  \inheritsavedanchors[from=rectangle] 
	  \inheritanchorborder[from=rectangle]
	  \inheritanchor[from=rectangle]{center}
	  \inheritanchor[from=rectangle]{north}
	  \inheritanchor[from=rectangle]{south}
	  \inheritanchor[from=rectangle]{west}
	  \inheritanchor[from=rectangle]{east}
	  \inheritanchor[from=rectangle]{north east}
	  \inheritanchor[from=rectangle]{south east}
	  \inheritanchor[from=rectangle]{north west}
	  \inheritanchor[from=rectangle]{south west}
	  \backgroundpath{
	    \southwest \pgf@xa=\pgf@x \pgf@ya=\pgf@y
	    \northeast \pgf@xb=\pgf@x \pgf@yb=\pgf@y
	    \pgfkeysgetvalue{/tikz/rectangle with rounded corners north west}{\pgf@rectc}
	    \pgfsetcornersarced{\pgfpoint{\pgf@rectc}{\pgf@rectc}}
	    \pgfpathmoveto{\pgfpoint{\pgf@xa}{\pgf@ya}}
	    \pgfpathlineto{\pgfpoint{\pgf@xa}{\pgf@yb}}
	    \pgfkeysgetvalue{/tikz/rectangle with rounded corners north east}{\pgf@rectc}
	    \pgfsetcornersarced{\pgfpoint{\pgf@rectc}{\pgf@rectc}}
	    \pgfpathlineto{\pgfpoint{\pgf@xb}{\pgf@yb}}
	    \pgfkeysgetvalue{/tikz/rectangle with rounded corners south east}{\pgf@rectc}
	    \pgfsetcornersarced{\pgfpoint{\pgf@rectc}{\pgf@rectc}}
	    \pgfpathlineto{\pgfpoint{\pgf@xb}{\pgf@ya}}
	    \pgfkeysgetvalue{/tikz/rectangle with rounded corners south west}{\pgf@rectc}
	    \pgfsetcornersarced{\pgfpoint{\pgf@rectc}{\pgf@rectc}}
	    \pgfpathclose
	 }
	}
	\makeatother

	\tikzset{->-/.style={decoration={markings,mark=at position #1 with {\arrow{>}}},postaction={decorate}}}
	\tikzset{-<-/.style={decoration={markings,mark=at position #1 with {\arrow{<}}},postaction={decorate}}}

	\tikzstyle{every picture}=[baseline=-0.25em,scale=0.5]
	\pgfdeclarelayer{edgelayer}
	\pgfdeclarelayer{nodelayer}
	\pgfsetlayers{edgelayer,nodelayer,main}

	\tikzstyle{box} = [draw,shape=rectangle,inner sep=2pt,minimum height=6mm,minimum width=6mm,fill=white] 
	\tikzstyle{boxlarge} = [draw,shape=rectangle,inner sep=2pt,minimum height=1.5cm,minimum width=8mm,fill=white] 
	\tikzstyle{boxLarge} = [draw,shape=rectangle,inner sep=2pt,minimum height=2cm,minimum width=10mm,fill=white] 
	\tikzstyle{boxsmall} = [draw,shape=rectangle,inner sep=2pt,minimum height=3mm,minimum width=3mm,fill=white] 
	\tikzstyle{dot} = [inner sep=0mm,minimum width=3mm,minimum height=3mm,draw,shape=circle,text depth=-0.1mm]
	\tikzstyle{Zbwdot} = [dot, fill=\Zbwcolour]
	\tikzstyle{Xbwdot} = [dot, fill=\Xbwcolour]
	\tikzstyle{Ybwdot} = [dot, fill=\Ybwcolour]
	\tikzstyle{Wbwdot} = [dot, fill=\Wbwcolour]
	\tikzstyle{antipode} = [boxsmall] 

	\tikzstyle{state} = [draw, rectangle with rounded corners,
	  rectangle with rounded corners north west=8pt,
	  rectangle with rounded corners south west=8pt,
	  rectangle with rounded corners north east=0pt,
	  rectangle with rounded corners south east=0pt,
	,inner sep=2pt,minimum height=6mm,minimum width=6mm,fill=white]
	\tikzstyle{statelarge} = [draw, rectangle with rounded corners,
	  rectangle with rounded corners north west=8pt,
	  rectangle with rounded corners south west=8pt,
	  rectangle with rounded corners north east=0pt,
	  rectangle with rounded corners south east=0pt,
	,inner sep=2pt,minimum height=1.5cm,minimum width=8mm,fill=white]
	\tikzstyle{stateLarge} = [draw, rectangle with rounded corners,
	  rectangle with rounded corners north west=8pt,
	  rectangle with rounded corners south west=8pt,
	  rectangle with rounded corners north east=0pt,
	  rectangle with rounded corners south east=0pt,
	,inner sep=2pt,minimum height=2cm,minimum width=8mm,fill=white]
	\tikzstyle{effect} = [draw, rectangle with rounded corners,
	  rectangle with rounded corners north west=0pt,
	  rectangle with rounded corners south west=0pt,
	  rectangle with rounded corners north east=8pt,
	  rectangle with rounded corners south east=8pt,
	,inner sep=2pt,minimum height=6mm,minimum width=6mm,fill=white]
	\tikzstyle{scalar}=[diamond,draw,inner sep=1pt,font=\small,fill=white]

	\tikzstyle{cdnode}=[fill=white]
	\tikzstyle{labelnode}=[fill=white]
	\tikzstyle{tightlabelnode}=[fill=white,inner sep = 0.1mm]
	\tikzstyle{none}=[inner sep=0pt]
	\tikzstyle{whiteline}=[-, line width=4pt, draw=white]

	\tikzstyle{trace}=[circuit ee IEC,thick,ground,scale=2.5]
	\tikzstyle{cotrace}=[circuit ee IEC,thick,ground,rotate=180,scale=2.5]
	\tikzstyle{upground}=[circuit ee IEC,thick,ground,rotate=90,scale=2.5]
	\tikzstyle{downground}=[circuit ee IEC,thick,ground,rotate=-90,scale=2.5]

	\tikzstyle{doubled} = [line width=1.8pt] 
	
	\tikzstyle{empty diagram}=[draw=gray!40!white,dashed,shape=rectangle,minimum width=1cm,minimum height=1cm]

%% file: symbols/traceSym.tex
\begin{tikzpicture} [scale=1.2,transform shape, rotate = 0] 

\def\deltax{0.3} 
\def\deltay{0.5} 

\path[use as bounding box] (-\deltax,-0.7*\deltay) rectangle (\deltax,0.3*\deltay);

\node (mult) at (0,0.3*\deltay) [upground,scale=0.5] {};
\node (mult_label_in) at (0,-0.7*\deltay) {};
\draw[-] (mult_label_in) to (mult);

\end{tikzpicture}

%% file: newmacros.tex
\renewcommand{\DoubledCategory}[1]{\operatorname{DBL}\left(#1\right)}
\newcommand{\DoublingFunctor}[1]{\operatorname{dbl}\left[#1\right]} 

\newcommand{\PhiFoldedCategory}[2]{\operatorname{FLD}_{#1}\left(#2\right)}
\newcommand{\PhiFoldingFunctorSym}[1]{\operatorname{fld}_{#1}} 
\newcommand{\PhiFoldingFunctor}[2]{\PhiFoldingFunctorSym{#1}\left[#2\right]}

\newcommand{\CPMFunctorSym}{\operatorname{CPM}}
\newcommand{\CPMFunctor}[1]{\CPMFunctorSym\left[#1\right]} 
\newcommand{\PhiXiCPMCategory}[3]{\CPMFunctorSym_{#1,#2}\left(#3\right)}

\newcommand{\conjFunctor}[1]{{conj}_{#1}}

\renewcommand{\Automs}[2]{\operatorname{Aut}_{#1}\left(#2\right)}

\newcommand{\SymMonCatUniverse}{\operatorname{SMCs}}
\newcommand{\DagSymMonCatUniverse}{\operatorname{DagSMCs}}
\newcommand{\DagCompCatUniverse}{\operatorname{DagCompSMCs}}
\newcommand{\SMCUnivsCategory}{\operatorname{SMCUnivs}}

\newcommand{\UnderlyingSMCFunctor}[2]{\llbracket#1\rrbracket_{#2}}
\newcommand{\UnderlyingSMC}[1]{\llbracket#1\rrbracket_{#1}}

\newcommand{\PreCPMFunctor}[1]{\operatorname{PreCPM}\left[#1\right]}

\newcommand{\norm}[1]{\operatorname{N}\left(#1\right)}

%% file: symbols/ZbwcomultSym.tex
\begin{tikzpicture} [scale=1.2,transform shape,rotate=0] 

\def\deltax{0.3} 
\def\deltay{0.5} 


\node (mult_label_outl) at (-\deltax,+\deltay) {};
\node (mult_label_outr) at (+\deltax,+\deltay) {};
\node [dot, fill=\Zbwcolour] (mult) at (0,0) {};
\node (mult_label_in) at (0,-\deltay) {};
\draw[-] [in=270,out=135] (mult) to (mult_label_outl);
\draw[-] [in=270,out=45] (mult) to (mult_label_outr);
\draw[-] (mult_label_in) to (mult);

\end{tikzpicture}

%% file: symbols/ZbwdotSym.tex
\begin{tikzpicture} [scale=1.2,transform shape] 

\def\deltax{0.3} 
\def\deltay{0.5} 


\node [dot, fill=\Zbwcolour] (mult) at (0,0) {};

\end{tikzpicture}